\documentclass[10pt]{amsart}

\usepackage[margin=1.5in]{geometry}                
\usepackage{graphicx}
\usepackage{amssymb}
\usepackage{epstopdf}
\usepackage{color}
\usepackage{xcolor}
\usepackage{ulem}
\usepackage{enumerate}
\usepackage{mathtools}
\usepackage{enumitem}
\usepackage{amsmath}
\usepackage{braket}
\usepackage{mathrsfs}  
\usepackage{hyperref}

\usepackage [english]{babel}
\usepackage [autostyle, english = american]{csquotes}
\MakeOuterQuote{"}

\newcommand{\R}{\mathbb{R}}

\newtheorem{fact}{Theorem}[section]

\newtheorem{obs}[fact]{Observation}

\newtheorem{defn}[fact]{Definition}

\newtheorem{remark}[fact]{Remark}

\title{$x$ Plays Pok\'emon, for almost-every $x$.}

\date{16 Jun 2025}
\begin{document}

\maketitle
\begin{center}
\author{C. Evans Hedges}
\end{center}

\begin{abstract} This paper provides a brief write-up showing that for any finite state game, a disjunctive number $x$ will eventually win that game. The proof techniques here are well known and this result follows immediately from folklore results in automata theory. This short paper primarily serves as an expositional piece to collect this proof with the fun context of $\pi$ Plays Pok\'emon serving as motivation. 
\end{abstract}

\section{Introduction}

In October 2021, the Twitch channel {\it WinningSequence} began streaming $\pi$ plays Pok\'emon, where each digit $0-9$ was mapped to a GameBoy Advance button, and $\pi$ began playing Pok\'emon Sapphire at one digit per second. As of writing, {\it WinningSequence} has played for just over $\pi$ years, has pushed over $100$ million buttons, and despite having never left the starting area has a level 90 Sceptile. This naturally leads us to the question: will $\pi$ eventually win Pok\'emon? 

Although $\pi$ is suspected of being a normal number \cite{borel1909}, we do not know if this is the case and unfortunately because of our lack of knowledge of the properties of $\pi$, we are unable to answer this question at present. However, we are able to show that for any disjunctive number $x$ \cite{calude-zamfirescu1998} (which is Lebesgue almost-every number \cite{borel1909}), $x$ will eventually win Pok\'emon. The theorem is stated for any finite state game with finitely many actions that can be taken at any point in time. It also extends to games with determinsitic pseudorandom number generators and holds with probability $1$ for games with true random number generators. This result follows from a well known result in automata theory \cite{cerny1964,volkov} and this paper aims to be a fun introduction to how graph and automata theory can be used to model games that we might find interesting.

\section{Background}

In order to model these finite state games (of which all computer games must necessarily be due to memory constraints) we start with a collection $V$ of possible game states. If the game additionally uses a pseudo-random number generator (pRNG) from a single generated seed, we may also include in $V$ the current state of the pRNG. Thus, given a state of the game $v \in V$, for a given button push $b$, $v$ will deterministically transition to some $v' \in V$. We then say there is a directed edge from $v$ to $v'$ with a label of $b$. Let $\mathcal{E}$ denote the collection of all edges for a given game. We then say the game $G$ is the graph $G = (V, \mathcal{E})$. In this sense, a game contains information about the state of the game, as well as the possible transitions between states. 

Note that this kind of game is deterministic since we have embedded the current state of the pRNG into $V$. If we did not, we would have to include either multiple transitions with the same label combined with the output of the (p)RNG. For this paper, however, we stick to the simplified case where $V$ includes the state of the pRNG and our game state transitions can be thought to be deterministic. We note here that this is the case for the classic Pok\'emon games, including Pok\'emon Sapphire, which uses a Linear Congruential Generator (LCG) \cite{knuth1997} with a period of $2^{32}$. Thus, the state space $V$ is effectively the product of the game's memory states and the $2^{32}$ possible states of the LCG.

We also say that a path in the graph $G$ is a sequence of edges $e_1 e_2 \dots e_k$ that traverse a sequence of vertices $v_0 v_1\dots v_{k}$ where the edge $e_i$ is an edge from $v_{i-1}$ to $v_i$. For a given graph $G$ as described above, a word $w$ is called synchronizing \cite{cerny1964} if for any two starting vertices $v_1, v_2 \in V$, if we traverse the path of edges generated by $w$ (i.e., the edges that are labeled according to the labels from $w$), the resulting vertex is the same. I.e. $v_1 \mapsto_{w} v$ and $v_2 \mapsto_{w} v$.

Finally, a vertex $v \in V$ is called a sink state if for every edge $e \in \mathcal{E}$ that leaves $v$, $e$ maps $v$ to $v$. This means that once you have reached $v$, all future states of the game will be $v$ and the game is effectively stuck. In the context of gaming, these are often referred to as "hard-locks." For the remainder of this paper we will assume that $G$ has a singular sink state $\tilde{v}$ such that for any $v \in V$ there exists a path from $v$ to $\tilde{v}$. In terms of the game, we want $\tilde{v}$ to denote that you have won the game, and requiring that any $v \in V$ has a path to $\tilde{v}$ means that no matter what the current state of the game, it is possible to still win the game. This means that we are assuming the game is "win-reducible"—you cannot hard-lock yourself in a state where winning becomes impossible. 

Finally, we consider how the button pushes are generated to play the game. For our given game $G$, we have some finite number of actions $b$ that are possible. We denote these actions as symbols from the alphabet $\mathcal{B} = \{0, \dots, b-1\}$. Thus, we can take any real number $x \in \R$ and examine the base-$b$ expansion of $x$. An important concept for this paper is that of a disjunctive number. 

\begin{defn}
We say $x \in \R$ is disjunctive in base $b$ if for any finite word $w \in \{0, \dots, b-1\}^n$, the word $w$ appears in the base-$b$ expansion of $x$. We say $x$ is disjunctive if it is disjunctive for every base $b \geq 2$.
\end{defn}
Note here that Lebesgue almost every real is a normal number \cite{borel1909}, which is a much stronger property that requires every finite word to occur {\it with the appropriate frequency}. We do not define this formally here since it is not required, but we simply note that almost every number is disjunctive \cite{calude-zamfirescu1998}.

\section{Statement of Theorems}

\begin{fact}\label{graph fact} Let $G$ be a game (graph) as described above. Then there exists a synchronizing word $w$ \cite{cerny1964} that maps all vertices to $\tilde{v}$, bounded above in length by $|G|^2$ \cite{pin1983}. 
\end{fact}

\begin{remark} The quadratic bound $|G|^2$ is a standard result \cite{pin1983}, but it is worth noting the connection to the famous \v{C}ern\'y Conjecture \cite{cerny1964}. The conjecture posits that for any synchronizing automaton with $n$ states, there exists a synchronizing word of length at most $(n-1)^2$ \cite{volkov}. While this remains one of the most significant open problems in automata theory, the existence of \textit{some} polynomial bound is sufficient for our purposes.
\end{remark}

\begin{proof} First we enumerate the states $V = \{ \tilde{v}, v_1, \dots, v_N \}$, where $N = |V|-1$. For each $1 \leq i \leq N$, there exists a winning word $w_i$ mapping $v_i$ to $\tilde{v}$. We construct the synchronizing word $W$ inductively. 

Set $W_1 = w_1$. After applying $W_1$, we observe that $v_1$ is mapped to $\tilde{v}$. Since $\tilde{v}$ is a sink state, any further inputs will leave the game in state $\tilde{v}$. Now consider $v_2$. After the application of $W_1$, the state $v_2$ is mapped to some $v_{i_2} \in V$. Let $w_{i_2}$ be the word that maps $v_{i_2}$ to $\tilde{v}$. We then set $W_2 = W_1 w_{i_2}$. By construction, $W_2$ maps both $v_1$ and $v_2$ to $\tilde{v}$. Continuing this process for all $N$ vertices, we obtain a finite word $W = W_N$ such that for every starting vertex $v \in V$, the game state after application of $W$ is $\tilde{v}$. 

Finally, we note that each $w_i$ can be chosen to have length at most $|V|$. Thus, the length of $W$ is bounded above by $N|V| \leq |V|^2$.
\end{proof}

\begin{remark} While Theorem \ref{graph fact} guarantees the existence of a winning word, finding the shortest such word is a significantly more difficult task. It is a known result in computational complexity that finding the shortest synchronizing word for a given automaton is NP-hard \cite{eppstein1990,olschewski-ummels2010}. Thus, while $x$ will win eventually, a speedrunner (or a particularly impatient real number) would find the task of optimizing this win sequence to be computationally intractable.
\end{remark}

\begin{fact} Let $x$ be disjunctive. Then $x$ will win any deterministic game that cannot be lost.  
\end{fact}

\begin{proof} This follows immediately from the Theorem \ref{graph fact}. Let $W$ be the synchronizing word that exists by Theorem \ref{graph fact}. Since $W$ will appear in $x$ and we know that $W$ will win regardless of the current state of the game, it does not matter what $x$ does before seeing the $W$, after completing $W$ we are guaranteed to win the game. 
\end{proof}

\begin{obs} If a game has $b$ buttons to push, then we have $b^{|G|^2}$ sequences of length $|G|^2$. If we concatenate all of these, we then beat the game after a sequence of $|G|^2 b^{|G|^2}$ button pushes. 
\end{obs}

This bound is horribly bad for real world considerations, as Pok\'emon Sapphire has a save state of $1$MB, which means we can bound $|G|$ by $8\times10^6$. Since it also has $8$ buttons, we can win Pok\'emon Sapphire in 
$$(8\times 10^6)^2 \times 8^{(8\times 10^6)^2} = 2^{6} \times 10^{12} \times 2^{3 \times 2^{6} \times 10^{12}} \approx  2^{10^{14}}$$
button pushes. This is a horrible bound; this number has $\approx 30$ trillion digits, and with there being approximately $10^{80}$ atoms in the observable universe, this number is effectively $10^{30 \text{ trillion}}$ times larger than that. It is similarly large compared to the number of Planck volumes in the observable universe. This bound is horrifyingly bad and even though we can write $2^{10^{14}}$ relatively easily here, it is an insanely large number for any practical purpose.

%

As an additional fact, we note here that if we appropriately encode our game $G$ to support RNG separately, so long as the RNG will show every word with some positive probability then every disjunctive number $x$ will win $G$ with probability $1$ (with respect to the RNG). 

\begin{fact} Let $G$ be a game with RNG. Let the RNG have full support over the space of all words. If $x$ is disjunctive then $x$ will win the game $G$ with probability $1$ (with respect to the RNG). 
\end{fact}

\begin{proof} Let $W$ be the synchronizing word from Theorem \ref{graph fact} and let $R$ be the required sequence of RNG outputs for which $(W, R)$ is guaranteed to win. Since the RNG has full support with respect to all finite words, the probability $P(R) > 0$, and since $x$ is disjunctive we know $W$ occurs infinitely often. Thus, for each occurrence of $W$, there is probability $(1-P(R)) < 1$ that $x$ will not win the game at that point. Since $W$ occurs infinitely often, we know in the limit the probability that $x$ will not win is $\lim_{n \rightarrow \infty} (1-P(R))^n = 0$. Thus, with probability $1$, $x$ will win the game. 
\end{proof}

\section{Concluding Remarks}

We have shown that for almost-every real number $x$, the base-$b$ expansion of $x$ contains a sequence of button pushes that will win any deterministic, win-reducible game. Because a disjunctive number contains every finite string, it will also necessarily contain the sequence of inputs for every possible Tool-Assisted Speedrun (TAS), including the frame-perfect optimal playthrough (assuming the buttons are pushed at the appropriate speed). 

However, as shown by the bound in Observation 3.1, this result is purely existential. The length of the guaranteed winning sequence for a game as complex as Pok\'emon Sapphire is so vast that it exceeds the number of Planck volumes in the observable universe by many orders of magnitude. While a disjunctive number $x$ will win eventually, it is almost certain that the physical hardware running the game would succumb to the heat death of the universe long before the winning digit is reached. Thus, while $\pi$ likely contains a winning playthrough of Pok\'emon, it is sadly quite possible that the universe does not contain enough time for us to witness it.

%
%
%
%
%
%
%
%
%
%
%
%
%

%
%
%
%

\bibliographystyle{plain}
\bibliography{mybib}

\end{document}